\documentclass[12pt]{amsart}
\usepackage[top=1.2in, bottom=1.2in, left=1.5in, right=1.5in]{geometry}
\usepackage{bbm}

\usepackage{amsfonts, amssymb, amsmath, enumerate}

\numberwithin{equation}{section}

\DeclareMathOperator{\E}{\mathbb{E}}

\DeclareMathOperator*{\Span}{span}

\DeclareMathOperator{\HS}{HS}

\def \C {\mathbb{C}}
\def \N {\mathbb{N}}
\def \P {\mathbb{P}}
\def \R {\mathbb{R}}

\def \Z {\mathbb{Z}}

\def \EE {\mathcal{E}}

\def \e {\varepsilon}
\def \d {\delta}

\newcommand{\etc}{,\ldots,}

\newcommand{\pr}[2]{\left \langle #1, #2 \right \rangle}
\newcommand{\norm}[1]{\left \| #1 \right \|}

\newtheorem{theorem}{Theorem}[section]

\newtheorem{corollary}[theorem]{Corollary}
\newtheorem{lemma}[theorem]{Lemma}

\theoremstyle{remark}
\newtheorem{remark}[theorem]{Remark}

\title[Approximately Hadamard matrices and Riesz bases in  frames]{Approximately Hadamard matrices \\ and  Riesz bases in random frames}

\author{Xiaoyu Dong}
\author{Mark Rudelson}

\address{Department of Mathematics, University of Michigan, 530 Church St., Ann Arbor, MI 48109, U.S.A.}
\email{\{xydong, rudelson\}@umich.edu}
\thanks{Research supported in part by NSF grant  DMS 2054408 and by a fellowship from the Simons Foundation.}  

\date{\today}


\begin{document}
	
	\begin{abstract}
		An $n \times n$ matrix with $\pm 1$ entries which acts on $\R^n$ as a scaled isometry is called Hadamard. Such matrices exist in some, but not all dimensions. Combining number-theoretic and probabilistic tools we construct matrices with $\pm 1$ entries which act as approximate scaled isometries in $\R^n$ for all $n \in \N$. More precisely, the matrices we construct have condition numbers bounded by a constant independent of $n$.
		
		Using this construction, we establish a phase transition for the probability that a random frame contains a Riesz basis.
		Namely, we show that a random frame in $\R^n$ formed by $N$ vectors with  independent identically distributed coordinate having a non-degenerate symmetric distribution contains many Riesz bases with high probability provided that $N \ge \exp(Cn)$. On the other hand, we prove that if the entries are subgaussian, then a random frame fails to contain a Riesz basis with probability close to $1$ whenever $N \le \exp(cn)$, where $c<C$ are constants depending on the distribution of the entries.
	\end{abstract}
	
	\maketitle

\section{Introduction and main results}  \label{sec: intro}
Let $n<N$ be natural numbers. 
A set of vectors $X_1 \etc X_N \in \R^n$ is called a \emph{frame} if 
\begin{equation} \label{eq: frame}
	K(n,N) \norm{x}_2^2 
	\le \sum_{j=1}^N \pr{x}{X_j}^2 
	\le R K(n,N) \norm{x}_2^2 
\end{equation}
for all $x \in \R^n$. Here $R \ge 1$ is an absolute constant called the frame constant, and $K(n,N)>0$ is some function of $n$ and $N$. 
The notation $\norm{x}_2$ stands for the Euclidean norm of the vector $x=(x_1 \etc x_n)$:
\[
  \norm{x}_2= \left( \sum_{j=1}^n x_j^2 \right)^{1/2}.
\]

In the last 40 years, frame theory became a well-developed area of applied mathematics, see \cite{CKP}, \cite{Cr 1}, \cite{Cr 2}, and the references therein.
A frame can intuitively be regarded as overcomplete basis in $\R^n$. 
Because of this property, frames became a valuable tool in signal transmission.
A signal which is viewed as an $n$-dimensional vector can be encoded by the sequence of its inner products with the frame vectors. If this sequence is transmitted over a communication line, then the original signal can be reconstructed even if part of the coefficients is lost or corrupted in the process of transmission. 
Moreover, this encoding is robust, which means that if the inner products are evaluated with some noise, then the reconstructed version will be close to the original one with the error depending on the noise magnitude.

One of the most popular classes of frames in algorithmic applications is the set of random frames.
Such frames became also the method of choice in compressed sensing where one needs to reconstruct a low complexity signal from a small number of linear measurements, see, e.g., \cite{Ver}.
For example, if complexity is measured as the size of the support, and the support itself is unknown, the random frames provide robust recovery with optimal or almost optimal theoretical guarantees.

To construct a random frame, consider a random vector $X \in \R^n$ with centered uncorrelated coordinates of unit variance.  
In other words, assume that $\E X=0$ and $\E X X^\top = I_n$.
Let vectors $X_1 \etc X_N$ be independent copies of $X$.
The Law of Large Numbers implies that 
\[
  \lim_{N \to \infty}  \frac{1}{N} \sum_{j=1}^N X_j X_j^\top =I_n  \quad \text{a.s.}
\]
and thus, with probability close to $1$,
\[
	(1-\e) \norm{x}_2^2 
\le \sum_{j=1}^N \pr{x}{X_j}^2 
\le (1+\e)  \norm{x}_2^2   
\]
for all $x \in \R^n$ provided that $N=N(\e)$ is sufficiently large.

Another,  trivial way to construct a frame is to take several bases in $\R^n$ and concatenate them. 
This allows an exact reconstruction of the transmitted signal if the number of corrupted coordinates is relatively small.
Indeed, one can reconstruct the original vector from the set of transmitted coordinates for each basis separately and keep the copy which is repeated many times. 
While in practice the random frames perform better than such concatenated bases, it leads to a question whether a random frame contains a copy or copies of a nice basis. 
More precisely, a sequence of $n$ vectors $v_1 \etc v_n \in \R^n$ is called a Riesz basis if it possesses the frame property \eqref{eq: frame}.
This property ensures that the reconstruction is robust, i.e., that the reconstructed vector is close to the original one if the coordinates are distorted by adding a  small noise.
These considerations lead to a natural question of determining the values of $N$ for which a random frame $\{X_1 \etc X_N\} \subset \R^n$ contains one or many Riesz bases with high probability.

This problem can be conveniently translated to the language of random matrices. 
Namely, for an $n \times N$ matrix $A$, define its singular values as 
\[
s_{\max} (A) =s_1(A) \ge s_2(A) \ge \cdots \ge s_n(W) = s_{\min}(A) \ge 0,
\]
where $s_j(A)=\sqrt{\lambda_j(A A^\top)}$, and $\lambda_1(A A^\top) \etc \lambda_n(A A^\top)$ are eigenvalues of $A A^\top$ arranged in the decreasing order.
Also, define the condition number of $A$ as
\[
\kappa(A) = \frac{s_{\max}(A)}{s_{\min}(A)}
\]
using the convention that $\kappa(A)=\infty$ whenever $s_{\min}(A)=0$.
With this notation, the frame property \eqref{eq: frame} can be rewritten as $\kappa(A_{n,N}) \le C$ where $A_{n,N}$ is the $n \times N$ matrix with columns $X_1 \etc X_N$.
Thus, the problem of existence of a Riesz basis in a random frame can be recast as the question of existence of one or many well-conditioned square $n \times n$ submatrices of an $n \times N$ random matrix $A_{n,N}$ with i.i.d. entries.
Our first main result shows that the probability of finding such a submatrix undergoes a phase transition when $N$ is exponential in terms of $n$. 
Since the upper and the lower bound hold under somewhat different assumptions, we formulate them separately.

Denote by $[N]$ the set $\{1 \etc N\}$. 
Let $A$ be an $n \times N$ matrix.
If $I \subset [N]$, denote by $A_I$ the submatix of $A$ whose columns belong to $I$.
The following theorem shows that if $N$ is exponential in $n$, then with high probability, the $n \times N$ random matrix has many  square submatrices with  uniformly bounded condition numbers. In the language of frames, it means that a random frame with exponentially many vectors contains a large number of Riesz bases whose frame constants are uniformly bounded.
\begin{theorem}  \label{thm: exist}
	Let $A$ be an $n \times N$ matrix with i.i.d. symmetric non-degenerate entries. 
	Then there exist constants $c, C, \alpha, \beta>0$ depending on the distribution of entries of $A$ with the following property. \\	
	Assume that 
	\[
	 N \ge \exp(Cn).
	\] 
	Then there exists 
	\[
	  L \ge \exp(c n)
	\]
	such that
	\begin{align*}
	 & \P \big( \text{\rm exist disjoint subsets }
	  I_1 \etc I_L \text{\rm \  of } [N] \text{\rm \ with } \big. \\
	  &\qquad \big.  |I_j|=n \text{\rm \  and } \kappa(A_{I_j}) < \alpha \text{\rm \ for all } j \in [L] \big) \\
	  & \ge 1- \exp \left(  - \exp ( \beta n ) \right).
	\end{align*}
\end{theorem}

The strategy of proving Theorem \ref{thm: exist} relies on using  
 a certain deterministic $n \times n$ matrix $V$ having a bounded condition number. 
 Denote by $\text{Col}_j(M)$ the $j$-th column of the matrix $M$.
We partition the set of integers $[N]$ into $n$ subsets $I_1 \etc I_n$ of approximately the same size and show that with high probability, the set $\{\text{Col}_i(A) \}_{i \in I_j}$ contains many columns close to $\text{Col}_j(V)$.
Condition on the event that such columns exist and form an $n \times n$ matrix $B$ taking one column from each set.
Then conditionally this matrix can be viewed as a noisy version of the matrix $V$. 
This allows to show that with high probability, the matrix $B$ has a bounded condition number as well. 

The key to this strategy is a successful choice of the pattern matrix $V$.
Since we strive to prove Theorem \ref{thm: exist} under minimal assumptions, the choice of $V$ becomes a non-trivial task. 
Indeed, the requirement that a column of $A$ can be close to a column of $V$ with a non-negligible probability forces us to look for matrices $V$ whose entries are in the support of the distribution of an entry of $A$. 
The latter can be as small as two points, $a$ and $-a$ because of the symmetry assumption.
Thus, we need to to construct $V$ as a scaled copy of a matrix with $\pm 1$ entries.
Such matrices are known in some cases. 
For instance, the condition number of any Hadamard matrix is one.
An $n \times n$ matrix $H$ is called Hadamard if $n^{-1/2} H$ is an isometry.
The earliest result on the existence of Hadamard matrices was probably proved by Sylvester \cite{S-TI} who showed that Hadamard matrices exist for dimension $2^k$ where $k$ is any nonnegative integer (these matrices are now called Wash since their rows are Walsh functions).
Hadamard matrices is a well-studied subject, and a number of constructions of such matrices are available, see e.g., the books \cite{Ag}, \cite{Hor}, and the references therein.
In particular, Wallis \cite{JSW} proved that if $p>3$ is an integer, then there exists an Hadamard matrix of order $2^tp$, where $t = \left\lfloor {2{{\log }_2}(p - 3)} \right\rfloor $. Craigen \cite{C-SG} improved Wallis's result by showing that for any odd number $p$, there exists an Hadamard matrix of order $2^tp$, where $t = 4\left\lceil {\frac{1}{6}{{\log }_2}((p - 1)/2)} \right\rceil +2$. Recently, de Launey \cite{dL-AE} studied the asymptotic existence of Hadamard matrices and concluded that for any $\epsilon>0$, the set of odd numbers $k$ for which there is a Hadamard matrix of order $k2^{2+[ {\epsilon{{\log }_2}(k)} ]}$ always has positive density in the set of natural numbers.
Yet, the dimensions in which Hadamard matrices were constructed are rare.

This leads us to a task of constructing \emph{approximately Hadamard matrices}, i.e., matrices with $\pm 1$ entries and bounded condition numbers.
Some constructions of matrices with properties simlar to Hadamard's are available. For example, Banica, Nechita, and \.{Z}yczkowski \cite{BNZ-AH} defined an almost Hadamard matrix to be a $N$ dimensional real square matrix $H$, such that $H/\sqrt{N}$ is orthogonal, and is a local maximum of the $\ell_1$-norm of the entries on the orthogonal group $O(N)$. They showed the existence of almost Hadamard matrices under some special assumptions. There is also a notion of quasi-Hadamard matrix  \cite{BNZ-AH,KS-QH}, which is defined as a square matrix with $\{-1,1\}$ entries that maximizes the absolute value of the determinant, but there are only very limited results on the existence of those matrices. In summary, no existing construction is directly related to our purpose.

The second main result of the paper is the following theorem asserting the existence of an approximately Hadamard matrix in all dimensions.
\begin{theorem} \label{th: approx-Hadamard}
	There exist constants $0<c<C$  such that for any  $n \in \N$, one can find an $n \times n$ matrix $V$ with $\pm 1$ entries satisfying
	\[
	c \sqrt{n}  \le s_{\min}(V) \le  s_{\max}(V) \le C \sqrt{n}.
	\]	
\end{theorem}
The proof of Theorem \ref{th: approx-Hadamard} relies on Vinogradov's theorem from analytic number theory and combines number-theoretic and probabilistic ideas. 
The details are presented in Section \ref{sec: approx}. 
After Theorem \ref{th: approx-Hadamard} is proved, we prove Theorem \ref{thm: exist} in Section \ref{seq: submatrix}.

The conclusion of Theorem \ref{thm: exist} holds under minimal assumptions on the distribution of entries.
If we assume that the entries of the matrix are sub-gaussian, then the bound of Theorem \ref{thm: exist} becomes sharp.
Recall that a random variable $X$ is called subgaussian if there is $a>0$ such that 
\[
\E \exp \left( \frac{X^2}{a^2}\right) \le 2.
\]
If $X$ is subgaussian then the smallest number $a$ having this property is called the subgaussian norm of $X$ and denoted $\norm{X}_{\psi_2}$.
Subgaussian random variables form a large family containing many naturally arising ones, see, e.g. \cite{Ver book}.

The next  theorem shows that finding a submatrix with a bounded condition number requires an exponential number of columns for matrices with subgaussian entries.
\begin{theorem} \label{thm: not exist}
	Let $X$ be a centered subgaussian random variable.
	Then there exist $C,c, \tilde{c}, t_0>0$ depending only on $\frac{\norm{X}_{\psi_2}}{\norm{X}_2}$ with the following property.
	Let $t>t_0$, and assume that
	\[
	N \le \exp \left(  \frac{\tilde{c}}{t^4} n \right).
	\]
	Let $A$ be an $n \times N$ matrix whose  entries  are independent copies of  $X$. 
	Then 
	\[
	\P \left( \exists I \subset [N] \ |I|=n \text{ and } \kappa(A_I) < t \right) \le \exp \left(  - c \frac{n^2}{t^4}\right).
	\]
\end{theorem}

We prove  Theorem \ref{thm: not exist} in Section \ref{sec: no small}. Its proof is easier than that of Theorem \ref{thm: exist} and relies on the Hanson-Wright inequality \cite{RV-HW}.

\subsection*{Acknowledgment.} The second author is grateful to Marcin Bownik for helpful discussions and bringing his attention to the problem.  Part of this work was done when the second author visited the Weizmann Institute of Science.
He is grateful to the Institute for its hospitality.
The authors thank a referee for thoroughly checking the manuscript and correcting many typos.

\section{Approximately Hadamard matrices} \label{sec: approx}
 In this section we construct an $n \times n$ matrix with $\pm 1$ entries whose scaled copy acts on $\R^n$ as an approximate isometry.
 More precisely, for any sufficiently large $n$,  we construct an $n \times n$ matrix $V$ such that its condition number $\kappa(V)$
 is bounded by an absolute constant. 
 
 We use standard matrix norms below. Namely, $\norm{A}$ stands for the operator norm of an $n \times m$ matrix $A=(a_{i,j})$, and $\norm{A}_{\HS}$ stands for its Hilbert-Schmidt or Frobenius norm:
 \[
   \norm{A}= \max_{\norm{x}_2=1} \norm{Ax}_2, \quad \text{and} \quad 
   \norm{A}_{\HS}= \left( \sum_{i=1}^{n} \sum_{j=1}^{m} a_{i,j}^2 \right)^{1/2}.
 \]

 We will apply an above mentioned result of Wallis \cite{JSW} showing that Hadamard matrices exist in dimensions close to $n^3$.
 \begin{lemma} \label{lem: JSW}
 	There is $l_0 \in \N$ such that for any $l > l_0$, there exists an Hadamard matrix of dimension $m(l)$ with 
 	\[
 	m(l)=2^{2 \lceil \log_2(l-3) \rceil} l.
 	\]
 \end{lemma}
We will need the following corollary.
 \begin{corollary} \label{cor: JSW}
 	For any $\e>0$, there exists $N(\e)$ such that for any $n>N_0(\e)$, one can find an even number $m \in [(1-\e)n, (1+\e)n  ]$ for which there exists an Hadamard matrix of size $m \times m$.
 \end{corollary}

\begin{proof}
	For any $n>12$, there exists a unique $k \in \N$ such that $2^{2k} (2^{k-1}+3) < n \le 2^{2(k+1)} (2^{k}+3)$.
	Assume first that  $2^{2k} (2^{k-1}+3) < n \le 2^{2k} (2^{k}+3)$. Set 
	\[
	m=2^{2k} \left \lceil \frac{n}{2^{2k}} \right \rceil.
	\]
	By Lemma \ref{lem: JSW}, there exists an Hadamard matrix of size $m \times m$. Since 
	\[
	1 \le \frac{m}{n} = \frac{\lceil 2^{-2k} n \rceil}{2^{-2k} n},
	\]
	and $2^{-2k}n \ge 2^{k-1}+3 > (n/2)^{1/3}$, the result follows if we choose $N(\e)$ sufficiently large.
	
	Since the tensor product of Hadamard matrices is an Hadamard matrix, and there are Hadamard matrices of sizes $2 \times 2$ and $4 \times 4$, there exist Hadamard matrices of sizes $2 m(l)$ and $4 m(l)$ for all $l > l_0$.
	This allows completing the proof of the corollary in the remaining cases when $2^{2k+1} (2^{k-1}+3) < n \le 2^{2k+1} (2^{k}+3)$ and $2^{2k+2} (2^{k-1}+3) < n \le 2^{2k+2} (2^{k}+3)$.
\end{proof}

 The aim of this section is to construct \emph{approximately Hadamard} matrices in any dimension, i.e. matrices whose condition number is $O(1)$.
 To this end, we use a construction of approximately Hadamard matrices of a prime size.

	Let $q \in \N$ be an odd prime number. 
	For $k \in \Z_q$, denote
	\[
	  e_q(k)= \exp \left( 2 \pi i \frac{k}{q} \right).
	\]
Define the Fourier transform on $\Z_q$ setting 
\[
  \hat{v}(j)=  \sum_{k \in \Z_q} v(k) e_q(jk)
\]
for a vector $v \in \C^{\Z_q}$ and $j \in \Z_q$.
\begin{lemma}  \label{lem: prime}
	Let $q$ be an odd prime number. 
	Then there exists a vector $u_q \in \{-1,1\}^{\Z_q}$ such that 
	\[
	  \big| |\hat{u}_q(j)|- \sqrt{q} \big| \le\sqrt{q} \d_q \quad \text{for any } j \in \Z_q
	\]
		with $\d_q=C q^{-1/4} \sqrt{\log q}$.
\end{lemma}	
\begin{proof}
	 The construction closely follows the one in \cite[Proposition 3.2]{JSW}, which in turn originates in  \cite[Theorem 9.2]{MR}.
	 
	Let $v: \Z_q \to \{-1,1\}$ be the Legendre symbol (quadratic character mod $q$). 
	More precisely, let 
	\[
	  Q=\{k \in \Z_q: \ k=j^2 (\text{mod }q) \text{ for some } j \in \Z_q\} \setminus \{0\}
	\]
	be the set of quadratic residues, and set
	\[
	 v(k)=
	 \begin{cases}
	 	1, &\text{if } k \in Q; \\
	 	-1, &\text{if } k \in \Z_q \setminus (Q \cup \{0\}); \\
	 	0, &\text{if } k =0.
	 \end{cases}
	\]
	Then by a standard result on the Gauss sum, see e.g., \cite[Proposition 6.3.2.; p.71]{IR}, we have
	\[
	  |\hat{v}(j)|=
	  \begin{cases}
	  	\sqrt{q}, &\text{if }  j \in \Z_q \setminus \{0\}  \\
	  	0, &\text{if } j=0.
	  \end{cases}
	\]
	The difference between $v$ and the desired function $u_q$ is that $v(0)=0$ and $\hat{v}(0)=0$.
	We will perturb $v$ replacing some of its coordinates  by $-1$ to change the value of $\hat{v}(0)$ as required while keeping the other Fourier coefficients close to their original values. 
	To this end, consider a sequence of i.i.d. random variables $\{X_k\}_{k \in Q}$ such that
	\[
	 \P(X_k=-1)=q^{-1/2} \quad \text{and} \quad \P(X_k=1)=1-q^{-1/2}.
	\]
	Set
	\[
	u_q(k)=
	\begin{cases}
		X_k, &\text{if } k \in Q; \\
		-1, &\text{if } k \in \Z_q \setminus Q; \\
	\end{cases}
	\]
	Then $u_q: \Z_q \to \{-1,1\}$, so we only have to check the values of the Fourier coefficients.
	Let us start with the expectations. 
	We have
	\begin{align*}
	 \E \hat{u_q}(0)
	 =   \E \hat{u_q}(0) - \hat{v}(0) 
	 &=\sum_{k \in Q} (\E X_k -1 )-1= -2 q^{-1/2} |Q|-1 \\
	 &= {q^{ - 1/2}} - {q^{1/2}} - 1,
	 \intertext{and}
     \E \hat{u_q}(j) - \hat{v}(j) 
     &=\sum_{k \in Q} (\E X_k -1 )e_q(jk)-1
     = \sum_{k \in Q} (-2q^{-1/2})e_q(jk)-1
	\end{align*}
    for all  $j \in \Z_q \setminus \{0\}$.
    Evaluation of the last sum is standard, see \cite[Ch. 2]{Davenport}, or \cite[Ch. 6]{IR}. 
    Namely,
    \begin{align*}
    \left| 1 +	2\sum_{k \in Q} e_q(jk) \right|^2
    &= \left| \sum_{k \in \Z_q} e_q(jk^2) \right|^2
    =  \sum_{k,l \in \Z_q} e_q(jk^2) \overline{e_q(jl^2) } \\
    &= \sum_{k,l \in \Z_q} e_q \big(j(k+l)(k-l) \big)=q,
   \end{align*}
   where the last equality follows if we fix $k+l$ and sum over $k-l$ first.
    Thus, 
    \[
      \left| \E \hat{u_q}(j) - \hat{v}(j) \right| \le 2+ q^{-1/2}
      \quad \text{for any } j \in \Z_q \setminus \{0\},
    \]
    and so $\big| |\E \hat{u_q}(j)|-\sqrt{q} \big| \le 3$ for all $j \in \Z_q$.

    The quantity $ \hat{u_q}(j)  -  \E \hat{u_q}(j) $ is a linear combination of  i.i.d. centered random variables $X_k-\E X_k, \ k \in Q$ with coefficients $e_q(jk)$ whose absolute value is bounded by $1$. 
    Therefore, Bernstein's inequality yields
    \[
     \P \left( |  \hat{u_q}(j)  -  \E \hat{u_q}(j) | >t \right)
     \le 2 \exp \left( - c \min(t^2 q^{-1/2} , t) \right)
    \]
    for all $t>0$ and $j \in \Z_q$.
    Setting $t=Cq^{1/4} \sqrt{\log q}$ and taking the union bound over $j \in \Z_q$, we obtain
   \[
     \P \left( |  \hat{u_q}(j)  -  \E \hat{u_q}(j) |  \le C q^{1/4}  \sqrt{\log q} \text{ for all } j \in \Z_q \right)
     \ge 1-  q^{-1}>0
   \] 
   if the constant $C>0$ is chosen sufficiently large. 
   The lemma follows.
\end{proof}

\begin{corollary} \label{cor: circulant}
	Let $q$ be an odd prime number. There exists a $q \times q$ matrix $U_q$ with $\pm 1$ entries such that
	\[
	  \sqrt{q} (1- \d_q) \le s_{\min}(U_q) \le  s_{\max}(U_q) \le \sqrt{q} (1+ \d_q)
	\]
	with $\d_q=C q^{-1/4} \sqrt{\log q}$.
\end{corollary}
\begin{proof}
	Represent $\Z_q$ as $\{1 \etc q\}$, and let $u_q: \{1 \etc q\} \to \{-1,1\}$ be the vector defined in Lemma \ref{lem: prime}.
	Let $U_q$ be the circulant matrix with the first row $u_q$. 
	A circulant matrix is  diagonal in the Fourier basis, see e.g., \cite[Theorem 3.2.1; p. 72]{Dav}.
	 Therefore, the singular values of $U_q$ are the absolute values of its eigenvalues which are the Fourier coefficients of the generating vector $u_q$.
	The result follows from Lemma \ref{lem: prime}.
\end{proof}

\begin{remark} \label{rem: U_q}
		Corollary \ref{cor: circulant} implies that the matrix $U_q$ satisfies
		\begin{align} \label{eq: U_q}
			\norm{ U_q (U_q)^\top -q I_q } \le 3 \delta_q q.
		\end{align}
	Inequality \eqref{eq: U_q} will be used later in the proof of Theorem \ref{th: approx-Hadamard}.
	\end{remark}

With this auxiliary construction in place, we can  prove the main result of this section, namely Theorem \ref{th: approx-Hadamard}.

\begin{proof}[Proof of Theorem \ref{th: approx-Hadamard}]
	The proof of this theorem combines a deterministic construction of number-theoretic nature with a probabilistic argument.
	Without loss of generality, we can assume that $n$ is larger than some number $n_0$ chosen in advance. 
	Indeed, after the statement of the theorem is proved for $n \ge n_0$, we can adjust the constants $c$ and $C$ appropriately to make it hold for all $n \in \N$.
	
	We start with the case when $n$ is even.
	Let $\e>0$  be a number to be chosen later.
	A combination of the Prime Number Theorem and Vinogradov's sum of three primes theorem \cite{Nat}, yields that there exists $N=N(\e)$ such that any even $n>N$ has a decomposition 
	\begin{equation}  \label{eq: four primes}
	  n=q_1+q_2+q_3+q_4 \text{ with } (1-\e) \frac{n}{4}\le q_j \le (1+\e) \frac{n}{4},
	\end{equation}
	where $q_1 \etc q_4$ are prime numbers.
	Indeed, by the Prime Number Theorem, there exists a prime number $q_1$ such that  $(1-\e/2) \frac{n}{4} \le q_1 \le (1+\e/2) \frac{n}{4}$.
	Then $m=n-q_1$ is odd, and thus by a stronger version of Vinogradov's theorem, it can be decomposed as
	\begin{equation} \label{eq: three primes}
	  m=q_2+q_3+q_4, \text{ where } (1-\e/2) \frac{m}{3} \le q_j \le (1+\e/2) \frac{m}{3},
	\end{equation}
    and $q_2,q_3, q_4$ are primes.
	This immediately implies \eqref{eq: four primes}.
	Actually, decompositions with bounds tighter than \eqref{eq: three primes} are available. 
	More precisely, one can find a representations such as \eqref{eq: three primes} with $|q_j-m/3|<m^\theta$ for some $\theta \in (0,1)$, see e.g., \cite{BH,Has,MMS}.
	However, the weaker version presented above will be sufficient for our purposes.
	
	Without loss of generality, assume that $q_1 \ge \cdots \ge q_4 =:q$.	
	We will consider the case $q_3>q_4$ first. This is the most non-trivial case, and the other ones will be treated in the same way after obvious modifications.
	For $j=\{1 \etc 4\}$, 
	let $U_j$ be the matrix $U_{q_j}$ constructed in Corollary \ref{cor: circulant}, and denote by $U_j^{top}$ the submatrix formed by the $q$ top rows of $U_j$. 
	For $j \in \{1,2,3\}$, denote by $U_j^{bottom}$ the submatrix of $U_j$ formed by its $q_j-q$ bottom rows.
	We will construct the matrix $W=V^\top$ in the following block form:
	\[
	  W=
	  \begin{pmatrix}
	  	W_{1,1} & W_{1,2}  & W_{1,3}  & W_{1,4} \\
	  	\vdots & \ddots  &   & \vdots \\
	  	\vdots &  & \ddots  &\vdots \\
	  	W_{4,1} & W_{4,2}  & W_{4,3}  & W_{4,4} \\
	  	W_{5,1} & W_{5,2}  & W_{5,3}  & W_{5,4} \\
	  	W_{6,1} & W_{6,2}  & W_{6,3}  & W_{6,4} \\
	  	W_{7,1} & W_{7,2}  & W_{7,3}  & W_{7,4} \\
	  \end{pmatrix}
       = \begin{pmatrix}
       	W^{top} \\W^{bottom}
       \end{pmatrix},
	\]
	where the matrix $W^{top}$ consists of the upper $4$ block rows of $W$, and $W^{bottom}$ consists of the lower three.
	Here $W_{j,k}$ is a $q \times q_k$ matrix if $1 \le j,k \le 4$ and a $(q_{j-4}-q) \times q_k$ matrix if $j=5,6,7, \ 1 \le k \le 4$.

	Let us define the matrices $W_{j,k}$.
	The matrix $W^{top}$ will be deterministic, and the matrix $W^{bottom}$ will consist of deterministic and random blocks.
	For $1 \le j,k \le 4$, set $W_{j,k}= \e_{j,k} U_k^{top}$, where $\e_{j,k}, \ j,k \in \{1 \etc 4\}$ form a $4 \times 4$ Walsh matrix:
	\[
	  \begin{pmatrix}
	  	\e_{1,1} & \e_{1,2} & \e_{1,3} & \e_{1,4} \\
	  	\e_{2,1} & \e_{2,2} & \e_{2,3} & \e_{2,4} \\
	  	\e_{3,1} & \e_{3,2} & \e_{3,3} & \e_{3,4} \\
	  	\e_{4,1} & \e_{4,2} & \e_{4,3} & \e_{4,4}
	  \end{pmatrix}
  =
      \begin{pmatrix}
      	1 & 1 & 1 & 1 \\
      	1 & -1 & 1 & -1 \\
      	1 & 1 & -1 & -1 \\
      	1 & -1 & -1 & 1 \\
      \end{pmatrix}.
	\]
	Now, let us define the matrices $W_{j,k}$ for $j=5,6,7$. 
	Set $W_{j, j-4}=U_{j-4}^{bottom}$.
	For $j=5,6,7$ and $k \neq j-4$, let $W_{j,k}$ be a random matrix with i.i.d. Rademacher entries.
	
	Consider the $(4q) \times (4q)$ matrix $W^{top} (W^{top})^\top$ first. 
	The diagonal blocks of this matrix are close to $nI_q$.
	More precisely,  for any $j \in \{1 \etc 4\}$, 
	\begin{align*}
	 \norm{ \sum_{k=1}^4 W_{j,k} W_{j,k}^\top - n I_q }
	 &= \norm{ \sum_{k=1}^4 \left( U^{top}_k (U^{top}_k)^\top -q_k I_q \right)} 
	 \le \sum_{k=1}^4 \norm{   U_k (U_k)^\top -q_k I_{q_k} } \\
	 &\le 12 \delta_q n,
	\end{align*}
	where the first inequality follows since $U^{top}_k (U^{top}_k)^\top $ is a submatrix of $ U_k (U_k)^\top $ and the second one from \eqref{eq: U_q}.
	
	Let us consider the off-diagonal blocks now.
	If $i \neq j, \ i,j \in \{1 \etc 4\}$ then similarly
	\begin{align*}
	  \norm{ \sum_{k=1}^4 W_{j,k} W_{i,k}^\top   }
	  &= \norm{ \sum_{k=1}^4 \e_{i,k} \e_{j,k}  U^{top}_k (U^{top}_k)^\top  } \\
	  &\le  \norm{ \sum_{k=1}^4 \e_{i,k} \e_{j,k}  \left( U^{top}_k (U^{top}_k)^\top -q_k I_q \right)  } + \left| \sum_{k=1}^4 \e_{j,k}  \e_{i,k}  q_k \right|\\
	  &\le 12 \delta_q n + \e n,
    \end{align*}
    where the first estimate follows from the triangle inequality, and the second one from $q_k \in [(1-\e) \frac{n}{4}, (1+\e) \frac{n}{4} ]$.
    Combining the two inequalities, we obtain
    \begin{equation}  \label{eq: W-top}
       \norm{  W^{top} (W^{top})^\top - n I_{4q} } \le 12 \d_q n+ 12 ( 12 \delta_q n + \e n)
       \le 13 \e n
    \end{equation}
     for all sufficiently large $n$. 
    
    Let us introduce auxiliary $(n-4q) \times n$ matrices $S$ and $R$ defined by
    \[
      S=
      \begin{pmatrix}
      	U_1^{bottom}& 0 & 0 &0 \\
      	0 & U_2^{bottom} &0 &0 \\
      	0 & 0 & U_3^{bottom} &0
      \end{pmatrix}
    \qquad R= W^{bottom} - S.
   \] 
   In other words, $R$ is the random part of the matrix $W^{bottom}$, i.e.,
   \begin{align*}
   	R=
   	\begin{pmatrix}
   		0 & W_{5,2}  & W_{5,3}  & W_{5,4} \\
   		W_{6,1} & 0  & W_{6,3}  & W_{6,4} \\
   		W_{7,1} & W_{7,2}  & 0  & W_{7,4} \\
   	\end{pmatrix}
\end{align*}
 is an $(n - 4q) \times n$ matrix with zeros along the block diagonal corresponding to the positions of $U_1^{bottom}$, $U_2^{bottom}$, and $U_3^{bottom}$  and i.i.d. Rademacher entries elsewhere.
   
   Recall that $n-4q \le 4 \e n$.
   In view of Corollary \ref{cor: circulant} and inequality \eqref{eq: U_q}, 
   \begin{align}  \label{eq: S and R}
   \norm{S} 
   &\le \sqrt{\frac{n}{4}} (1+\d_q),
   \qquad
   \norm{S (W^{top})^\top} 
   \le 12 \d_q n, \\
     \norm{S}_{\HS}^2 
     &\le 3 \e n \cdot (1+\e)\frac{n}{4}
     \le \e n. \notag
   \end{align}
   Also,
   \begin{align*}
     \norm{S S^\top - \frac{n}{4} I_{n-4q}}
     &\le \max_{j=1,2,3} \norm{U_j^{bottom}(U_j^{bottom})^\top - q_j I_{q_j-q}} + \e \frac{n}{4} \\
     &\le  \e \frac{n}{2}.
   \end{align*}

   Let $\tilde{R}$ be an $(n-4q) \times n$ matrix with i.i.d. Rademacher entries.
   Then a simple symmetrization argument yields
   \[
     \P \left( \norm{S R^\top} \ge C \sqrt{\e} n  \right)
     \le 2 \P \left( \norm{S \tilde{R}^\top} \ge (C/2) \sqrt{\e} n  \right)
   \]
   which in combination with  \cite[Theorem 3.2]{RV-HW} implies that
   \[
    \P \left( \norm{S R^\top} \le C \sqrt{\e} n  \right)
    \ge 1 - \exp(-c \e n).
   \]
   Another application of symmetrization yields
   \[
     \P(\norm{R} \ge 4 \sqrt{n})
     \le 2 \P(\|\tilde{R}\| \ge 2 \sqrt{n})
     \le \exp(-cn).
   \]
   Fix a matrix $R$ for which 
   \begin{equation}  \label{eq: SR}
     \norm{S R^\top} \le C \sqrt{\e} n 
     \quad \text{and } \norm{R} \le 4 \sqrt{n}
   \end{equation}
   at the same time.
   
   Let $x \in S^{n-1}$. 
   Following the previous convention, we write 
   \[
     x=
     \begin{pmatrix}
     	x^{top} \\ x^{bottom}
     \end{pmatrix},
   \]
   where $x^{top} \in \R^{4q}$ and $x^{bottom} \in \R^{n-4q}$.
   Assume first that
    $\norm{x^{bottom}}_2 \ge \eta$, where the constant $\eta>0$ will be chosen below.
   Then
   \begin{align*}
   	\norm{W^\top x}_2 
   	&\ge \frac{1}{\norm{S}} \cdot \norm{S W^\top x}_2 \\
   	&\ge  \frac{1}{\norm{S}} \cdot \left( \norm{S (W^{bottom})^\top x^{bottom}}_2 - \norm{S (W^{top})^\top x^{top}}_2 \right) \\
   	&\ge (1- 2\d_n) \sqrt{\frac{4}{n}} \cdot \\
   	& \qquad  \left( \norm{S S^\top x^{bottom}}_2 - \norm{S R^\top x^{bottom}}_2 - \norm{S (W^{top})^\top x^{top}}_2 \right) \\
   	&\ge (1- 2\d_n) \sqrt{\frac{4}{n}} \cdot \\
   	& \qquad \qquad \left( (1 - 2 \e) \frac{n}{4} \norm{x^{bottom}}_2 -  \norm{S R^\top } - \norm{S (W^{top})^\top } \right) \\
   	&\ge (1- 4 \d_n) \sqrt{\frac{n}{4}} \cdot  \left( (1 - 2 \e) \norm{x^{bottom}}_2 - C_1 \sqrt{\e} \right) \\
   	&\ge \frac{\eta}{4} \sqrt{n}
   \end{align*}
   if $\eta$ and $\e$ are chosen so that $(1 - 2 \e) \eta - C_1 \sqrt{\e} >\eta/2 $.

   Assume now that $\norm{x_{bottom}}_2 < \eta$.
   Then $\norm{x_{top}}_2 > 1- \eta$, and   \eqref{eq: W-top} yields
   \begin{align*}
   	\norm{W^\top x}_2
   	&\ge \norm{(W^{top})^\top x^{top}}_2 - \norm{(W^{bottom})^\top} \cdot \norm{x^{bottom}}_2 \\
   	&\ge (1-7 \e) \sqrt{n} \cdot \norm{x^{top}}_2  - C_2 \sqrt{n}  \cdot \norm{x^{bottom}}_2 \\
   	&\ge (1-7 \e) \sqrt{n} \cdot (1-\eta)  - C_2 \sqrt{n}  \cdot \eta \\
   	&\ge \frac{1}{2} \sqrt{n}
   \end{align*}
   if $\eta$ is chosen so that $C_2 \eta < \frac{1}{4}$.
  Choosing the parameters $\e$ and $\eta$ sufficiently small, we can 
    reconcile the two restrictions, i.e., select $\e, \eta$ so that the inequalities 
   \[
    (1 - 2 \e) \eta - C_1 \sqrt{\e} >\eta/2 \quad \text{and} \quad  C_2 \eta < \frac{1}{4}
   \]
   hold at the same time.
   With this choice, the previous argument shows that 
   \[
     \norm{W^\top x}_2
     \ge \min \left( \frac{\eta}{4}, \frac{1}{2}\right) \sqrt{n}
   \]
   for all $x \in S^{n-1}$, which means that 
   \[
     s_{\min} (W^\top) \ge c \sqrt{n}.
   \]
   Obtaining a bound for $s_{\max}(W)$ is easier. Inequalities \eqref{eq: W-top} and \eqref{eq: S and R} imply
   \[
     \norm{W^{top}} \le 2 \sqrt{n} \quad \text{and } \norm{S} \le \sqrt{n}.
   \]
    This in combination with \eqref{eq: SR} yields
   \[
     s_{\max} (W^\top) \le C \sqrt{n},
   \]
  which proves the theorem in the case $q_3>q$.
   
   If $q_j=q$ for some $j \in \{1,2,3\}$, then we repeat the same argument with the block rows of $W$  containing $q_j-q=0$ rows removed. 
   For instance, if $q_1>q$ and $q_2=q_3=q$, then we consider the $n \times n$ matrix $W$ with $W^{top}$ being the same as in the previous case and $W^{bottom}=(W_{5,1} \cdots W_{5,4})$ which is a $(q_1-q) \times n$ matrix.
   
   This completes the proof of the theorem in the case when $n$ is even.
   
   \vskip 0.1in
   
   Assume that $n$ is odd.
   By Corollary \ref{cor: JSW}, if $n$ is sufficiently large, then we can find an even number 
   \[
     m \in \left[ \left(1- \frac{\e}{2} \right) \frac{n}{4},  \left(1+ \frac{\e}{2} \right) \frac{n}{4} \right]
   \]
   for which there exists an Hadamard matrix $V$ of size $m \times m$.
   Note that $n-m$ is odd and 
    \[
      n-m \in \left[ \left(1- \frac{\e}{2} \right) \frac{3n}{4},  \left(1+ \frac{\e}{2} \right) \frac{3n}{4} \right],
   \]
    so using Vinogradov's theorem again, we obtain a decomposition
   \[
   n-m=q_1+q_2+q_3,
   \]
   where $q_1, q_2, q_3$ are prime numbers and $\frac{1-\e}{4} n \le q_j \le  \frac{1+\e}{4} n$.
   At this point we can apply the same argument we used in the case of an even $n$ with one of the matrices $U_1 \etc U_4$ replaced by $V$.
   This completes the proof of the theorem.
\end{proof}

\section{Submatrices with a small condition number}  \label{seq: submatrix}

 In this section we prove Theorem \ref{thm: exist}. 
 As was explained in the Introduction, the proof relies on finding  columns of $A$ which are close to columns of a scaled copy of the matrix $V$ constructed in the previous section.
 Conditioned on the event that such selection is possible, we prove that with high probability, the constructed submatrix has a bounded condition number.
 We start with the latter task, namely with analyzing a random matrix close to $V$.
 \begin{lemma}  \label{lem: close to V}
 	Let $V$ be an $n \times n$ matrix with $\pm 1$ entries such that 
 	\[
 	  c_{\ref{lem: close to V}} \sqrt{n}  \le s_{\min}(V) \le  s_{\max}(V) \le C_{\ref{lem: close to V}} \sqrt{n}
 	\]	
 	for some $0< c_{\ref{lem: close to V}} \le  C_{\ref{lem: close to V}}$.
 	
 	There exists $\d \in (0,1)$ for which
 	any  $n \times n$ matrix $Y$ with i.i.d. entries $Y_{i,j}$ such that 
 	\[
 	 \E Y_{i,j}=0 \quad \text{and } |Y_{i,j}| \le \d \text{ a.s.}
 	\] 
 	satisfies
 	\[
 	  \P \left( \kappa(V+Y) \le 4 \frac{C_{\ref{lem: close to V}}}{c_{\ref{lem: close to V}}}\right) 
 	  \ge 1 - \exp (-cn).
 	\]
 \end{lemma}

 \begin{proof}
 	The proof of Lemma \ref{lem: close to V} uses the basic net argument, see e.g., \cite{Ver book}.
 	Since $Y$ has i.i.d. centered subgaussian entries with $\norm{Y_{i,j}}_{\psi_2} \le \norm{Y_{i,j}}_{\infty} \le \d$,
 	\[
 	 \P (\norm{Y} \ge C \sqrt{\d n}) \le \exp(-cn).
 	\]
 	Therefore,
 	\[
 	 \P ( s_{\max}(V+Y) \ge 2  C_{\ref{lem: close to V}} \sqrt{n})
 	 \le
 	 \P (\norm{V} +\norm{Y} \ge 2  C_{\ref{lem: close to V}} \sqrt{n}) \le \exp(-cn),
 	\]
 	as we can always assume that $ C_{\ref{lem: close to V}} \ge 1$ and choose $\d$ sufficiently small.
 	Similarly,
 	\[
 	\P ( s_{\min}(V+Y) \le \frac{1}{2}  c_{\ref{lem: close to V}} \sqrt{n})
 	\le
 	\P (s_{\min}(V) -\norm{Y} \le \frac{1}{2}  c_{\ref{lem: close to V}} \sqrt{n}) 
 	\le \exp(-cn),
 	\]
 	where as before, the last inequality holds for any sufficiently small $\d$.
 	The result follows by combining the two bounds above.
 \end{proof}

We now proceed to proving the main result, Theorem \ref{thm: exist}.
 
 \begin{proof}[Proof ot Theorem \ref{thm: exist}]
 	Since the distribution of entries of $A$ is non-degenerate, there exists $a>0$ such that for any $\nu>0$
 	\[
 	  \P \left(|a_{i,j}-a| < \nu \right) >0.
 	\]
 	By the symmetry of distribution, we also have the same property for $-a$ in place of $a$.
 	
 	 	Let $\d>0$ be as in Lemma \ref{lem: close to V}, and denote
 	 $
 	   \nu= \frac{a}{4} \d.
 	 $
 	Let $Z$ be a random variable having the same distribution as $a_{i,j}$ conditioned on the event that $|a_{i,j}-a | \le \nu$.
 	More precisely, for a Borel set $E \subset \R$, set 
 	\[
 	  \P(Z \in E)= \frac{1}{ \P(|a_{i,j}-a | \le \nu)} \P(a_{i,j} \in E \ \& \ |a_{i,j}-a | \le \nu).
 	\]
 	Set
 	\[
 	 R=\frac{Z- \E Z}{\E Z}.
 	\]
 	Then $R$ is a centered random variable such that
 	\[
 	  |R| \le \frac{2 \nu}{(1-\d/4)a} \le \d \ \text{ a.s.}
 	\]

 	Let $C>0$ be a constant to be chosen later, and assume that $N \ge \exp(2 C n)$.
 	Partition $[N]$ into a union of sets $I_1 \etc I_n$ such that 
 	\[
 	 |I_j| \ge \left\lfloor \frac{N}{n} \right\rfloor  \ge \exp( C n).
 	\]
 	Let $V$ be the $n \times n$ matrix with $\pm 1$ entries constructed in Theorem \ref{th: approx-Hadamard}.
 	Denote its columns by $V_1 \etc V_n$ and the columns of $A$ by $A_1 \etc A_N$.
 	Let $M$ be a number to be chosen later.
 	Let $\EE$ be the event that for any $j \in [n]$, there exist at least $M$ numbers $k \in I_j$ with
 	\[
 	 \norm{A_k- a V_j}_{\infty} \le  \nu.
 	\]
 	If $\EE$ occurs, denote by $k(j, 1) \etc k(j, M)$ the first $M$ numbers $k \in I_j$ having this property. 
 	Then  conditioned on $\EE$, for any $m \in [M]$, the matrix $A_{\EE, m}$ with columns $A_{k(1, m)} \etc A_{k(n, m)}$  has the same distribution as  $\E Z \cdot (V +Y)$, where $Y$ is an $n \times n$ random matrix whose entries have the form $Y_{i,j}=V_{i,j} R_{i,j}$, where $R_{i,j}$ are independent copies of $R$.
 	In view of Lemma \ref{lem: close to V}, this implies that for any $m \in [M]$,
 	\[
        \P \left( \kappa(A_{\EE, m}) \le 4 \frac{C_{\ref{lem: close to V}}}{c_{\ref{lem: close to V}}}  \mid \EE \right) 
         \ge 1 - \exp (-cn).
    \] 	
    Since conditionally on $\EE$, the matrices $A_{\EE,1} \etc A_{\EE,M}$ are independent, Bernstein's inequality allows to conclude that
    \[
        \P \left( \kappa(A_{\EE, m}) \le 4 \frac{C_{\ref{lem: close to V}}}{c_{\ref{lem: close to V}}}  \text{ for at least } M/2 \text{ numbers } m \in [M] \mid \EE \right) 
        \ge 1 - \exp (-c'M).   
    \]
    To complete the proof, we have to show that the probability of $\EE^c$ is small. 
    To this end, denote $\eta= \P(|a_{i,j}-a | \le \nu)$. 
    Then by the symmetry of distribution of  the entries of $A$, $\P(|a_{i',j'}-a V_{i,j}| \le \nu) =\eta$ for any $i',j'$.
    Let $j \in [n]$. 
    For any $k \in I_j$, 
    \[
       \P( \norm{A_k- a V_j}_{\infty} \le  \nu )
       = \eta^n.
    \]
    Set 
    \[
       M= \frac{N}{2} \cdot \eta^n 
       = \frac{1}{2} \exp \left(  Cn - \log  \left( \frac{1}{\eta} \right) \cdot n \right) 
       \ge \exp \left( \frac{C n}{2}\right)
    \]
    where the last inequality holds if $C=C(\eta)$ is chosen sufficiently large. 
    Note that the events $\norm{A_k- a V_j}_{\infty} \le  \nu$ are independent for all $k \in I_j$.
    At this point, Bernstein's inequality yields
    \begin{align*}
       \P( \norm{A_k- a V_j}_{\infty} \le   \nu \text{ for less than } M \text{ numbers } k \in I_j    )
      & \le \exp (-c''M) \\
       &\le \exp \left[ - \exp \left( \frac{C n}{4}\right)\right].
    \end{align*}
    Therefore,
    \begin{align*}
    	\P(\EE^c)
    	&\le \sum_{j=1}^{n} \P( \norm{A_k- a V_j}_{\infty} \le   \nu \text{ for less than } M \text{ numbers } k \in I_j    ) \\
    	&\le n \cdot  \exp \left[ - \exp \left( \frac{C n}{4}\right)\right] 
    	\le  \exp \left[ - \exp \left( \frac{C n}{8}\right)\right].
    \end{align*}
  Let $L=M/2$, and $\alpha=4 \frac{C_{\ref{lem: close to V}}}{c_{\ref{lem: close to V}}} $. 
  Combining the previous inequalities, we obtain that
 	\begin{align*}
      	 & \P \big( \text{\rm exist disjoint subsets }
      I_1 \etc I_L \text{\rm \  of } [N] \text{\rm \ such that } \big. \\
      &\qquad \big.  |I_j|=n \text{\rm \  and } \kappa(A_{I_j}) < \alpha \text{\rm \ for all } j \in [L] \big) \\
      &\ge   \P \left( \kappa(A_{\EE, m}) < \alpha \text{ for at least } M/2 \text{ numbers } m \in [M] \mid \EE \right)  \cdot (1-\P(\EE^c)) \\
      &\ge 1 - \exp(-c'M) -  \exp \left[ - \exp \left( \frac{C n}{8}\right)\right] \\
      & \ge 1- \exp \left(  - \exp ( \beta n ) \right)
    \end{align*}
for an appropriate $\beta>0$.
This completes the proof of the theorem.
 \end{proof}

 \section{No submatrices with a small condition number} \label{sec: no small}

  In this section, we prove  Theorem \ref{thm: not exist}.
  \begin{proof}
  	  Without loss of generality, we can assume that $\norm{X}_2 = (\E X^2 )^{1/2}=1$.
  	Throughout the proof, we denote by $C,c,c'$, etc. constants depending only on $\norm{X}_{\psi_2}$.
  	
    Consider an $n \times n$ random matrix $B$ whose entries are independent copies of $X$.
    We claim that 
    \begin{equation}  \label{eq: norm B}
    	\P (\norm{B} \le c \sqrt{n}) \le \exp (-c' n^2).
    \end{equation}
    Indeed, denoting the columns of $B$ by $B_1 \etc B_n$, and applying the Hanson-Wright inequality \cite[Theorem 2.1]{RV-HW}, we get
    \begin{align*}
     &\P (\norm{B} \le \frac{1}{2} \sqrt{n}) \le  \P (\norm{B_j}_2 \le  \frac{1}{2} \sqrt{n} \text{ for all } j \in [n]) \\
     &\le \left( \P\left[ \E \norm{B_j}_2^2  - \norm{B_j}_2^2  \ge  \frac{1}{2} n \right] \right)^n 
      \le \exp (-c' n^2).
    \end{align*}
    Furthermore, we assert that 
    \begin{equation}  \label{eq: smallest B}
	  \P (s_n(B) \ge 2 \sqrt{\e n}) \le \exp (-c' \e^2 n^2)
    \end{equation}  
for any $\e>0$. 
 Proving \eqref{eq: smallest B} relies on a standard fact from linear algebra.
\begin{lemma}  \label{lem: s_n upper}
	Let $M$ be an $n \times n$ matrix. 
	Let $k<n$, and denote $H= \Span (Me_{k+1} \etc Me_n)$, where $e_1 \etc e_n$ is the standard basis of $\R^n$.
	Then
	\[
	  s_n(M) \le \min_{j=1 \etc k} \norm{P_{H^\perp}  M e_j}_2,
	\]
	where $P_{H^\perp}$ is the orthogonal projection on $H^\perp$.
\end{lemma}
\begin{proof}
	Without loss of generality, we can assume that the matrix $M$ is invertible.
	Let  $j \in [k]$. 
	Choosing an appropriate $u \in \Span (e_{k+1} \etc e_n)$, we obtain
	\[
	  1 \le \norm{e_j-u}_2 
	  \le \norm{M^{-1}} \cdot \norm{Me_j- Mu}_2 
	  = s_n^{-1} (M) \cdot \norm{P_{H^\perp}M e_j}_2,
	\]
	where the equality holds after optimization over $u$.
	The lemma follows.
\end{proof}

 To prove \eqref{eq: smallest B}, we apply Lemma \ref{lem: s_n upper} to $B$ setting $k = \lfloor \e n \rfloor$. 
 It yields
 \[
    \P (s_n(B) \ge 2 \sqrt{\e n}) \le \P (\norm{P_{H^\perp}  B e_j}_2   \ge 2 \sqrt{\e n}  \text{ for all } j \in [k] ).
 \]
 Conditioning on $B_{k+1} \etc B_n$, we can rewrite the right hand side ov the above inequality as
 \begin{align*}
    &\E  \left( \P \big[ \norm{P_{H^\perp}  B e_j}_2   \ge 2 \sqrt{\e n}  \text{ for all } j \in [k]  \mid B_{k+1} \etc B_n \big] \right) \\
    =  &\E  \left(\P \big[  \norm{P_{H^\perp}  B e_1}_2   \ge 2 \sqrt{\e n}   \mid B_{k+1} \etc B_n \big] \right)^k
 \end{align*}
 using independence of the columns of $B$.
  The conditional probability can be estimated by applying the Hanson-Wright inequality again. 
  Applying  \cite[Theorem 2.1]{RV-HW} to the vector $B_1=Be_1$ having i.i.d. centered subgaussian coordinates, we get
  \begin{align*}
   &  \P \big[  \norm{P_{H^\perp}  B e_1}_2   \ge 2 \sqrt{\e n}   \mid B_{k+1} \etc B_n \big]  \\
    \le \ & \P \big[  \norm{P_{H^\perp}  B e_1}_2  - \norm{P_{H^\perp}}_{\HS} \ge  \sqrt{\e n}   \mid B_{k+1} \etc B_n \big] \\
    \le \ & \exp (-c \e n).
  \end{align*}
  Taking the expectation with respect to $B_{k+1} \etc B_n$ and combining it with the previous inequality completes the proof of \eqref{eq: smallest B}.
  
  Using \eqref{eq: smallest B} with $\e=t^{-2}/4$ together with \eqref{eq: norm B}, we derive
  \[
   \P(\kappa(B) < t) \le \exp \left( - c' \frac{n^2}{t^4}\right).
  \]
  The proposition follows by using this inequality for $B=A_I$ and taking the union bound over $I \subset [N]$:
  \begin{align*}
    \P \left( \exists I \subset [N] \ |I|=n \text{ and } \kappa(A_I) < t \right) 
    &\le \binom{N}{n}  \exp \left(  - c' \frac{n^2}{t^4}\right) \\
    &\le \exp \left( n \log \left( \frac{e N}{n}\right)  - c' \frac{n^2}{t^4} \right) \\
    &\le \exp \left(   - (c'-\tilde{c}) \frac{n^2}{t^4} \right),
  \end{align*}
where the last inequality follows from the assumption on $N$.
Setting $\tilde{c}=c'/2$ completes the proof.  
  \end{proof}

\end{document}